\documentclass[10pt,a4paper]{article}

\usepackage{latexsym}
\usepackage{verbatim}
\usepackage{amsthm}
\usepackage{amsfonts}
\usepackage{amsmath}
\usepackage{amssymb}
\usepackage{amscd}

\numberwithin{equation}{section}

\newtheorem{prop}{Proposition}[section]
\newtheorem{theorem}[prop]{Theorem}
\newtheorem{lemma}[prop]{Lemma}
\newtheorem{corollary}[prop]{Corollary}
\newtheorem{remark}[prop]{Remark}
\newtheorem{example}[prop]{Example}
\newtheorem{definition}[prop]{Definition}

\def\begeq{\begin{equation}}
	\def\endeq{\end{equation}}

\def\<{\langle}
\def\>{\rangle}
\def\({\left(}
\def\){\right)}

\def\pa{\partial}
\def\e{\varepsilon}

\def\Om{\Omega}
\def\we{\wedge}

\def\s{\sigma}

\def\lbr{\lbrace}
\def\rbr{\rbrace}
\def\P{\mathcal P_p}
\def\M{\mathcal M_p}

\DeclareMathOperator{\la}{\lambda}

\begin{document}
	
	\title{Interior estimates for $p$-plurisubharmonic functions}
	\author{S\l awomir Dinew }

	\date{}
	\maketitle
	
	\begin{abstract} We study a Monge-Amp\`ere type equation in the class of $p$-pluri-\newline subharmonic functions and establish first and second order interior estimates. As an application of these we show that $p$-plurisubharmonic functions with constant operator and quadratic growth must be quadratic polynomials.
		
	\end{abstract}
	\section*{Introduction}
	
	Given a Riemannian manifold $(X,g)$, a $C^2$ smooth real valued function $u$ is said to be {\it subharmonic} if in local coordinates it satisfies the inequality
	$$\Delta_{g}(u):=\frac1{\sqrt{det(g)}}\sum_{k,l}\frac{\pa}{\pa x_k}(\sqrt{det(g)}g^{kl}\frac{\pa u}{\pa x_l})\geq 0,$$
	the definition being independent on the choice of local coordinates.
	
	Let now $X$ be a submanifold in the Euclidean space $\mathbb R^n$ equipped with the restriction of the Euclidean metric. Given a smooth function $u$ living in $\mathbb R^n$ it is a natural question when does it restrict to an intrinsically subharmonic one on $X$. 
	As a special
	case we note that a subharmonic function in $\mathbb R^n$ need not restrict subharmonically to an affine submanifold (take the line $\mathbb R\times\lbr0\rbr$ in $\mathbb R^2$ and the function $u(x,y)=-x^2+y^2$), while convex functions by definition restrict to convex, hence subharmonic 
	functions on any affine subspaces in $\mathbb R^n$.
	
	In these circle of ideas Harvey and Lawson \cite{HL13} gave a beautiful characterization of the class of functions that restrict subharmonically on {\it minimal submanifolds}:
	\begin{theorem}\label{Harvey-Lawson}[Harvey-Lawson]
		Let $\Om$ be a domain in $\mathbb R^n$ and $u$ be a $C^2$ function on $\Omega$. Then $u$ restricts subharmonically to any minimal $p$-dimensional submanifold of $\Omega$ ($1\leq p\leq n$) if and only if at each $x\in\Om$ the Hessian $D^2u(x)$ has eigenvalues $\la_1,\cdots,\la_n$ satisfying
		\begin{equation}\label{p-tuple}
			\forall 1\leq i_1<i_2<\cdots<i_p\leq n\ \ \ \la_{i_1}+\cdots+\la_{i_p}\geq 0.
		\end{equation}
		
	\end{theorem}
	Note that such functions, called {\it $p$-plurisubharmonic}, interpolate
	between subharmonic ones (in the case $p=n$) and the convex ones for $p=1$. Similar notions can be defined on any Riemannian manifold, again using the geometric fact that functions restrict to subharmonic ones on $p$-dimensional minimal submanifolds- see \cite{HL13}. Analogous notion is also meaningful in the complex setting- \cite{HL13} with the geometric interpretation that functions restrict subharmonically
	to $p$-dimensional {\it complex} submanifolds.
	
	In \cite{HL13} and later in \cite{HL1,HL2} Harvey and Lawson developed a non- linear potential theory associated to the class of $p$-plurisubharmonic functions.  As it turns out many results from the complex {\it pluripotential theory} (see \cite{K2} for an overview)
	have their $p$-plurisubharmonic counterparts- we refer
	to the paper \cite{HL12} for a nice survey. A major difference though was the lack of a natural differential operator playing the role of a (nonlinear) Laplacian just as the complex Monge-Amp\`ere operator does in the pluripotential setting.
	
	One of the possibilities is to try to mimic the properties of the determinant and define and operator $\M(u)$ by
	$$\M(u(x))=\Pi_{1\leq i_1<i_2<\cdots<i_p\leq n}(\la_{i_1}+\cdots+\la_{i_p}),$$
	again $\la_i=\la_i(x)$ denoting the eigenvalues of the complex Hessian at $x$ (see Section 4 in \cite{HL1}). Such an operator, in the complex case, was discussed by Sadullaev (\cite{Sa}) and, in the special case $p=n-1$, by Tosatti and Weinkove \cite{TW}.
	
	The operator $\M(u)$, even though defined through eigenvalues, is in fact a special example of a fully nonlinear operator of Hessian type. Unfortunately, as shown in \cite{D}\footnote{In \cite{D} the complex analogue of $\M$ was invesitgated but the proof there carries through in the real setting.}
	the operator $\M$ does not satisfy an {\it integral comparison principle} which makes the associated potential theory much harder. On the bright side from the very definition it is clear that $\M(u)\geq 0$ for any $p$-plurisubharmonic function.
	It turns out that $\M$ is also an {\it elliptic} operator, when 
	restricted to the class of $p$-plurisubharmonic functions, while $\widetilde{\mathcal M}_p:=\M^{1/\binom np}$ is a concave operator. All this implies that the nondegenerate Dirichlet problem for $\M$ is solvable in smoothly bounded strictly convex domains as a special case
	of the Caffarelli-Nirenberg-Spruck theorem (\cite{CNS}) shows:

	\begin{theorem}\label{CafNirSpruck}[Caffarelli-Nirenberg-Spruck] Let $\Om$ be a bounded strictly convex domain in $\mathbb R^n$ with $\pa\Om\in C^{3,1}$ and $\psi\in C^{3,1}(\pa \Om)$ be given. Then for any $f\in C^{1,1}(\Omega\times\mathbb R)$, such that $f\geq c>0$ the Dirichlet problem
		\begin{equation}\label{cns}
			\begin{cases}
				u- p-{\rm plurisubharmonic\ in}\ \Om\cap C(\overline{\Om});\\
				\M(u)=f(x,u)\ \ {\rm in}\ \ \Om;\\
				u|_{\pa \Om}=\psi
			\end{cases}
		\end{equation}
		admits a unique solution $u$, which is $C^{1,1}$ up to the boundary.
	\end{theorem}
	
	The aim of this note is to investigate the interior regularity and Liouville type theorems for the $p$-plurisubharmonic solutions of the $\M$ equation. Following closely the
	arguments of Chou and Wang \cite{CW} (who dealt with a similar Hessian type equation) we prove the following {\it interior} $C^1$ and Pogorelov type {\it interior} $C^2$ estimates for $\M$:
	\begin{theorem}[First order interior estimate]
		Let $u$ be a $p$-plurisubharmonic function in the ball $B_r(x_0)$. Assume that $u\in C^3(B_r(x_0))\cap C^1(\overline{B_r(x_0)})$ and $u$ solves the equation
		$$\M(u)=f(x,u),$$
		for a given non negative Lipschitz function $f$. Then 
		\begin{equation*}
			|Du(x_0)|\leq C
		\end{equation*}
		for a constant $C$  dependent on $n, p, r$,  $sup_{B_r(x_0)}|u|$ and the Lipschitz bound on $f$. 
	\end{theorem}
	
	\begin{theorem}[Second order interior estimate]
		Let $\Om$ be a bounded domain in $\mathbb R^n$ and  the $p$-plurisubharmonic function $u\in C^4(\Omega)\cap C^{1,1}(\overline{\Om})$ solves the equation
		$$\M(u)=f(x,u),$$
		where $f\in C^{1,1}(\overline{\Om}\times\mathbb R)$ is given. Suppose that $f\geq f_0>0$, and that there is a $p$-plurisubharmonic function $w$, satisfying $w>u$ in $\Om$ with equality on $\partial\Om$. Then for any fixed $\delta>0$ there is a bound
		\begin{equation*}
			(w(x)-u(x))^{1+\delta}|D^2u(x)|\leq C,
		\end{equation*}
		where $C$ depends on $n,p, f_0,sup_{\Om}(|Du|+|Dw|)$ and $||f||_{C^{1,1}}$.
	\end{theorem}
	
	The techniques of Chou and Wang could be applied almost directly if a suitable algebraic control on the nonlinearities is established. This is possible for some pairs $(n,p)$ but in general the $\widetilde{\mathcal M}_p$ operator is not sufficiently concave to guarantee that. Instead, we shall apply the by now standard trick of estimating the largest eigenvalue of the Hessian. Thus, roughly speaking, the main input of the
	current note is that the {\it positivity cones} associated to the $p$-plurisubharmonic
	functions share similar convexity properties to the cones $\Gamma_k$ studied in \cite{CW} (see also \cite{Wa2}- Section 2). Analogous interior estimates for the {\it complex} analogue of $\M$ remain an open problem.

	Regarding further regularity we show that $\widetilde{\mathcal M}_p$ is a concave operator and hence the standard Evans-Krylov theory applies (see \cite{GT}). Roughly speaking this  yields that $C^{1,1}$ solutions have to be in $C^{2,\alpha}(\Om)$ for some $\alpha>0$.
	
	As a
	result routine arguments (see \cite{HSX,LS,Y} for example)  yield the following Liouville type theorem:
	\begin{theorem}\label{Liouville}[Liouville theorem]
		Let $u$ be a $p$-plurisubharmonic function in $\mathbb R^n$ of quadratic growth. If
		$$\M(u)=const>0,$$
		then $u$ is a quadratic polynomial.
	\end{theorem}
	Here by quadratic growth we mean that for some constants $0<C<D$ and $|x|$ large we have $C|x|^2-C\leq u(x)\leq D|x|^2+D$. Such Liouville type results come in handy in analyzing blow-up profiles of solutions under rescalings. Our proof follows the lines of the one from \cite{HSX}, 
	where the Authors deal with entire solutions to $\sigma_2$ equation and assume only the lower quadratic bound on $u$. We do not know whether our theorem holds under this weaker assumption.
	
	{\bf Relation to previous work.} We wish to note that in a recent note Chu and Jiao \cite{CJ} dealt with a similar Hessian type equation (coinciding with $\M$ in the special case $p=n-1$) and were also able to prove a Pogorelov type interior $C^2$ estimate
	although their argument differs from ours (see Theorem 1.5 in \cite{CJ}). The interior gradient estimate for even broader set of equations has been established in \cite{De}.\footnote{The Author wishes to thank B. Deng for pointing out this reference which has been missed during the preparation of the manuscript.} When finishing the note we also learned about the unpublished preprint \cite{Z}. There, using probabilistic approach, the Dirichlet problem for a large class of degenerate elliptic equations
	is studied. In particular the Dirichlet problem (\ref{cns}) is shown to have a $C^{1,1}$ solution for a $u$-independent $f$ provided that $f\geq 0$ and $f^{1/\binom np}$ is $C^{1,1}$.
	
	The note is organized as follows: first we collect various algebraic and potential theoretic facts about the $\M$-operator and the class of $p$-plurisubharmonic functions which shall be used later on. In Section 2 we prove the first order estimate, while the second order estimate
	is dealt with in Section 3. In the last section we show the Liouville theorem.
	
	{\bf Acknowledgement}.
	The Author was supported by Polish National Science Centre grant 2017/26/E/ST1/00955. The Author wishes to thank Professor Weisong Dong for pointing out a mistake in the previous version of the note and for sharing ideas how to correct it. Thanks are also due to professor Bin Deng for informing me about his paper \cite{De} and professor Jianchun Chu for fruitful discussions on $\M$ equations.
	\section{Preliminaries}
	In this section we collect various conventions, definitions and properties that we shall rely on in the note. 
	\subsection{Notation and conventions}
	As it is customary by $C$ or $C_j$  we shall denote different uniform constants dependent on the relevant quantities that may vary line to line. Unless otherwise stated any tuple $(i_1,\cdots,i_p)$ of indices will be assumed to be ordered in increasing order. Eigenvalues of a matrix $A$
	will be denoted by $\la_j(A)$ or simply by $\la_j$ and will be assumed to ordered in a decreasing order except if the opposite is explicitly written. A summation over $k\in\lbr k_1,\cdots, k_p\rbr$ means that we sum over all (increasingly ordered) $p$-tuples 
	$(k_1,\cdots k_p)$ so that $k$ is among the $k_j$'s. Summation over $k\notin\lbr k_1,\cdots, k_p\rbr$ is defined analogously. Throughout the note we shall assume that $1<p<n$, as in the extreme cases the results we prove are well known.
	\subsection{Linear algebra of the $p$-convex cones}
	We begin with the algebraic details of the cones, called {\it $p$-convex cones} which  were introduced by Harvey and Lawson in \cite{HL13}:
	\begin{definition} Let $p\in\lbr1,\cdots,n\rbr$. The cone $\P$ is defined by
		$$\P=\lbr(\la_1,\cdots,\la_n)\in\mathbb R^n|\ \forall 1\leq i_1<i_2<\cdots<i_p\leq n,\ \la_{i_1}+\cdots+\la_{i_p}> 0\rbr.$$
	\end{definition}
	It is straightforward that $\P$ is a symmetric cone\footnote{Here symmetric cone means that $\P$ is invariant under permutations of the coordinates.} and $\mathcal P_p\subset\mathcal P_q$ whenever \newline $p<q$. Also $\mathcal P_1$ coincides with the {\it positive cone} 
	$$\Gamma:=\lbr\la\in\mathbb R^n|\ \forall j\in 1,\cdots,n\ \la_j> 0\rbr.$$
	
	In \cite{HL1} it is observed that 
	$$\forall p\in 1,\cdots, n\ \  \P+\Gamma\subset \P,$$
	hence in the language of \cite{HL1} the cones $\P$ are examples of elliptic {\it convex $ST$-invariant subequations} (see Sections 3 and 4 in \cite{HL1} for the details). This in turn allows to develop a rich nonlinear potential theory of {\it $p$-plurisubharmonic functions} i.e. functions with Hessian $D^2u$ having eigenvalues in $\P$.
	
	We briefly recall the details below.
	
	Associated to $\P$ is the $P_p$ cone of symmetric $n\times n$ matrices defined by
	$$P_p=\lbr A|\ \forall 1\leq i_1<i_2<\cdots<i_p\leq n,\ \la_{i_1}(A)+\cdots+\la_{i_p}(A)> 0\rbr,$$
	with $\la_j(A)$ denoting the $j$-th eigenvalue of $A$ ordered in a decreasing order. We denote by $\overline{P}_p$ the Euclidean closure of $P_p$ in the space of symmetric matrices.
	
	\begin{definition}\label{pplurisubharmonicfunction}
		Given a domain $\Om\subset\mathbb R^n$ a $C^2(\Om)$ function $u$ is said to be $p$-plurisubharmonic if for any point $x_0\in\Om$ one has
		$$D^2u(x_0)\in \overline{P}_p.$$
	\end{definition}
	An immediate corollary is that a $p$-plurisubharmonic function has to be subharmonic and that a $C^2$ smooth convex function is $p$-plurisubharmonic for any $p$. For more details we refer to \cite{HL13} and \cite{HL1}.
	
	\subsection{The operator $\M$}
	Our main interest in the current note will be the following operator:
	\begin{definition}\label{mpoperator}
		Let $u$ be a $C^2$ smooth function. Let $\la_j, \ j=1,\cdots, n$ denote the eigenvalues of the Hessian $D^2u$ at a point $x$. We define the operator $\M$ by
		$$\M(u(x)):=\Pi_{1\leq i_1<i_2<\cdots<i_p\leq n}(\la_{i_1}+\cdots+\la_{i_p}).$$
	\end{definition}
	
	Observe that the definition is symmetric with respect to the eigenvalues. Hence if 
	$$\sigma_k(u):=\sum_{1\leq i_1<i_2<\cdots<i_k\leq n}\la_{i_1}\cdots\la_{i_k}$$
	are the elementary symmetric polynomials of the eigenvalues the fundamental theorem of symmetric polynomials yields that $\M$ is expressible through $\sigma_j,\ j=1,\cdots,p$. Recall (see \cite{CW}) that $\sigma_k$ are examples of Hessian differential operators i.e. 
	second order differential operators invariant under orthonormal changes of basis in which the Hessian matrix is being computed. As a result also $\M$ is a Hessian type operator for any $p=1,\cdots, n$.
	
	The representation of $\M$ in terms of symmetric polynomials can be quite\newline  complicated as table below shows. The explicit formulas are computed in the cases $p\leq n\leq 5$ (we put $\M:=0$ if $p>n$).

	\begin{tabular}{llllll}
		& 1 & 2 & 3 & 4 & 5\\
		\hline
		1 & $\s_1$ & $\s_2$ & $\s_3$ & $\s_4$ & $\s_5$\\
		\hline
		2 & 0 & $\s_1$ & $\s_1\s_2-\s_3$ & $\s_1\s_2\s_3-\s_1^2\s_4$ & $\s_1\s_2\s_3\s_4-\s_1^2\s_4^2-\s_3^3\s_4$\\
		& & & &$-\s_3^2$ & $+2\s_1\s_4\s_5-\s_1\s_2^2\s_5$\\
		& & & & &$-\s_2\s_3\s_5-\s_5^2$\\
		\hline
		3 & 0 & 0 & $\s_1$ & $\s_1^2\s_2-\s_1\s_3$&$ (\s_1^2\s_2-\s_1\s_3+\s_4)(\s_1\s_2\s_3$\\
		& & & &$+\s_4 $ & $-\s_1^2\s_4-\s_3^2)+\s_1^5\s_5$\\
		& & & & & $-\s_1^3\s_2\s_5+3\s_1^2\s_3\s_5$\\
		& & & & & $-2\s_1\s_4\s_5+\s_2\s_3\s_5-\s_5^2$\\
		\hline
		4 & 0 & 0 & 0 & $\s_1 $&$ \s_1^3\s_2-\s_1^2\s_3+\s_1\s_4-\s_5$\\
		\hline
		5 & 0 & 0 & 0 & 0 &$ \s_1$
	\end{tabular}
	
	Nevertheless in the special case $p=(n-1)$ the formula is fairly explicit:
	\begin{prop}
		If $\mathcal M_p^{(n)}$ denotes the $\M$ operator in dimension $n$ then
		$$\mathcal M_{n-1}^{(n)}(\la)=\sum_{k=2}^n(-1)^k\s_1^{n-k}(\la)\s_k(\la).$$
	\end{prop}
	\begin{proof}
		Taking any vector $\la\in\mathbb R^n$ with $\la_n$=0 in 
		$$\Pi_{1\leq j_1<j_2<\cdots<j_{n-1}\leq n}(\lambda_{j_1}+\lambda_{j_2}+\cdots+\lambda_{j_{n-1}})$$
		results in the value of the operator $\s_1\mathcal M_{n-2}^{(n-1)}$ evaluated on the \newline  $(n-1)$-dimensional vector 
		of the first $(n-1)$ coordinates of $\la$. Simultaneously all terms involving $\s_n$ in the formula of $\mathcal M_{n-1}^{(n)}$ vanish. By the uniqueness of the representation by elementary symmetric polynomials one then obtains
		$$\mathcal M_{n-1}^{(n)}(\la)=\s_1(\la)\widehat{\mathcal M}_{n-2}^{(n-1)}(\la)+A\s_5(\la),$$
		where  $\widehat{\mathcal M}_{n-2}^{(n-1)}$ denotes the symmetric polynomial in $n$-variables with the representation in terms of elementary symmetric polynomials coinciding with the one of ${\mathcal M}_{n-2}^{(n-1)}$. The constant $A$ turns out to be equal to $(-1)^n$ as 
		evaluation on the vector $\la=(1,\cdots,1)$ shows. Finally the claimed formula follows from induction.
	\end{proof}
	
	Below we provide an alternative definition of $\M$ which will be utilized later on.
	
	\subsection{Definition through derivations on the exterior algebra}
	As observed by Harvey and Lawson in \cite{HL13} $p$-plurisubharmonicity and thus the $\mathcal M_p$-operator can be defined through derivations on the exterior algebra in the following way:
	
	Consider the space $\Lambda^p\mathbb R^n$. Fix an orthonormal basis $(e_1,\cdots, e_n)$  of $\mathbb R^n$ and the corresponding basis $e_{i_1}\we\cdots\we e_{i_p}$ of  $\Lambda^p\mathbb R^n$, where $(i_1,\cdots i_p)$ run over of all increasing $p$-tuples ordered in a lexicographical order. 
	
	Given a symmetric matrix $A$  we define the linear derivation of $A$ on  $\Lambda^p\mathbb R^n$ as the linear map defined by
	
	$$\mathcal D_A: \Lambda^p\mathbb R^n\ni(v_1\we\cdots\we v_p)\rightarrow (Av_1)\we v_2\we\cdots\we v_p+v_1\we(Av_2)\we v_3\we\cdots\we v_p$$
	$$+\cdots +v_1\we\cdots\we v_{p-1}\we(Av_p)\in \Lambda^p\mathbb R^n.$$
	
	Note that $\mathcal D_A$ is a symmetric endomorphism of  $\Lambda^p\mathbb R^n$. 
	
	We recall the crucial observation from \cite{HL13} which becomes obvious when one chooses as a basis for $\mathbb R^n$ the set of eigenvectors of $A$:
	\begin{prop}[Harvey-Lawson]
		$A$ has eigenvalues in $\P$ if and only if $\mathcal D_A$ is positive definite on  $\Lambda^p\mathbb R^n$.
	\end{prop}
	
	It is worth emphasizing that given any orthonormal basis of $\mathbb R^n$ $\mathcal D_A$ has a matrix representation with respect to the induced basis which has components being linear combinations of the entries of $A$. Below we compute two illustrative examples for a Hessian matrix of a function- the case of main interest in the note: 
	\begin{example}\label{a}
		For $n=3$ and $p=2$ and $A=D^2u(x)$  the corresponding matrix in the canonical basis $e_1\we e_2, e_1\we e_3, e_2\we e_3$ reads
		
		\begin{center}
			$\left[\begin{array}{lll}
				u_{11}+u_{22} & u_{23} & -u_{13}\\
				u_{32} & u_{11}+u_{33} & u_{12}\\
				-u_{31} & u_{21} & u_{22}+u_{33}
			\end{array}\right].$
		\end{center}
	\end{example}
	\begin{example}\label{b}
		For $n=4$ and $p=2$ and $A=D^2u(x)$ again in the canonical basis ordered in lexicographical order the matrix is
		\begin{center}
			$\left[\begin{array}{llllll}
				u_{11}+u_{22} & u_{23} & u_{24} & -u_{13} & -u_{14} & 0\\
				u_{32} & u_{11}+u_{33} & u_{34} & u_{12} & 0 & -u_{14}\\
				u_{42} & u_{43} & u_{11}+u_{44} & 0 & u_{12} & u_{13}\\
				-u_{31} & u_{21} & 0 & u_{22}+u_{33} & u_{34} & -u_{24}\\
				-u_{41} & 0 & u_{21} & u_{43} & u_{22}+u_{44} & u_{23}\\
				0 & -u_{41} & u_{31} & -u_{42} & u_{32} & u_{33}+u_{44}
			\end{array}\right].$
		\end{center}
	\end{example}

	A fundamental formula for us will be the following one linking the $\M$ operator and the above construction:
	\begin{lemma}\label{mpandda}
		Let $u\in C^2(\Om)$. Then for any $x\in\Om$ we have
		$$\M(u)(x)=det(\mathcal D_{D^2u(x)}).$$
	\end{lemma}
	\begin{proof} Observe that the determinant does not depend on the choice of an orthonormal basis of $\mathbb R^n$ yielding a basis for $\Lambda^p\mathbb R^n$. So we may choose a basis consisting of unit eigenvectors of $D^2u(x)$. But then obviously both sides are equal to
		\begin{equation*}
			\Pi_{1\leq i_1<i_2<\cdots<i_p\leq n}(\la_{i_1}+\cdots+\la_{i_p}).
		\end{equation*}
		
	\end{proof}
	\subsection{Formulas for $\M$ and its derivatives}
	To begin with we note that the equation $\M(u)=f$ for a function $u$ can be linearized as follows:
	\begin{lemma}\label{basicM}
		Let $u$ be a $C^2$ function. Then
		$$\sum_{k,l=1}^n\M^{kl}(u)u_{kl}=\binom np f,$$
		while (assuming the necessary higher order differentiability of $u$)
		$$\sum_{k,l=1}^n\M^{kl}(u)u_{klt}= \pa_tf,$$
		$$\sum_{k,l=1}^n\M^{kl}(u)u_{kltt}+\sum_{kl,rs}\M^{kl,rs}(u)u_{klt}u_{rst}=\pa_t\pa_tf,$$
		where $\M^{kl}(u):=\frac{\pa\M(u)}{\pa u_{kl}}$, $\M^{kl,rs}(u):=\frac{\pa^2\M(u)}{\pa u_{kl}\pa u_{rs}}$.
	\end{lemma}
	\begin{proof}
		The second and the third formulas simply follow from differentiating $\M(u)=f$ in the $\frac{\pa}{\pa t}$ direction. The first formula can be proved in various ways- for example using the expression of $\M$ in terms of elementary symmetric polynomials
		and utilizing the analogous formulas for $\sigma_k$.
		Alternatively we observe that we may without loss of generality diagonalize $D^2u$ at a fixed point and then according to the next lemma the claimed equality reads
		$$\M(u)\sum_{k=1}^n\sum_{k\in\lbr i_1,\cdots, i_p\rbr}\frac{u_{kk}}{u_{i_1i_1}+\cdots+u_{i_pi_p}}=\binom np f.$$
		The last equality then follows from elementary combinatorics.
	\end{proof}
	
	For our later purposes we need to compute the derivatives of $\M$ at a diagonal matrix $A$. 
	\begin{lemma}\label{Hessfla1}
		Let $A=(a_{kl})_{k,l=1}^n$ be a diagonal matrix with eigenvalues\newline $\la_1,\cdots,\la_n$. Then 
		$$\frac{\partial\M}{\partial a_{kl}}(A)=\M^{kl}(A)=\begin{cases}
			0\ \  {\rm if}\ \ l\neq k;\\
			\M(A)\sum_{l\in \lbr k_1,\cdots,k_p\rbr}\frac{1}{\la_{k_1}+\cdots+\la_{k_p}}\ \ {\rm if}\ \ l=k.\end{cases}$$
	\end{lemma}
	\begin{proof} This can be seen in various ways, for example using the expression of $\M$ in terms of Hessian $\sigma_k$ operators. Arguably the simplest way though is to use Lemma \ref{mpandda}
		and the formula for differentiation of determinants. Note that if $A$ is diagonal then so is $\mathcal D_{A}$ and thus immediately all derivatives with respect to non-diagonal entries have to vanish, while
		differentiation with respect to $a_{ll}$ affects only the diagonal terms of $\mathcal D_{A}$ containing it.
	\end{proof}
	As a direct corollary we obtain the following useful bound:
	\begin{corollary}\label{below}
		Id $A\in P_p$ is a diagonal matrix  and $\widetilde{\mathcal M}_p(A):=(\M(A))^{\frac1{\binom np}}$ then
		$$\sum_{k=1}^n\widetilde{\mathcal M}_p^{kk}(A)=\sum_{k=1}^n\frac{\pa\widetilde{\mathcal M}_p(A)}{\pa a_{kk}}\geq p.$$
	\end{corollary}
	\begin{proof}
		$\sum_{k=1}^n\widetilde{\mathcal M}_p^{kk}(A)$ is obviously equal to 
		$$\sum_{k=1}^n\frac1{\binom np}\M(A)^{\frac1{\binom np}-1}\M^{kk}(A)$$
		$$=\frac1{\binom np}\M(A)^{\frac1{\binom np}}\sum_{k=1}^n\sum_{k\in \lbr k_1,\cdots,k_p\rbr}\frac{1}{a_{k_1k_1}+\cdots+a_{k_pk_p}}$$
		$$=\frac{p}{\binom np}\M(A)^{\frac1{\binom np}}\sum_{\lbr1\leq k_1<\cdots<k_p\leq n\rbr}\frac{1}{a_{k_1k_1}+\cdots+a_{k_pk_p}}.$$
		An application of the AM-GM inequality bounds the last quantity from below by
		$p,$
		as claimed.
	\end{proof}

	The formulas for the second derivatives of $\M$ are a bit trickier and are given below: 
	\begin{lemma}\label{Hessfla2}
		Let $A$ be a diagonal matrix with eigenvalues $\la_1,\cdots,\la_n$. Then 
		\begin{equation*}\frac{\partial^2\M}{\partial a_{kl}a_{rs}}(A)=\M^{kl,rs}(A)=
		\end{equation*}
		$$=\begin{cases}
			\M(A)\sum\limits_{\substack {l\in \lbr j_1,\cdots,j_p\rbr\\ r\in\lbr i_1,\cdots, i_p\rbr\\  \lbr j_1,\cdots,j_p\rbr\neq \lbr i_1,\cdots,i_p\rbr}}\frac{1}{(\la_{j_1}+\cdots+\la_{j_p})(\la_{i_1}+\cdots+\la_{i_p})}\ \  {\rm if}\ \ k=l,\ r=s;\\
			-\M(A)\sum\limits_{\substack{l\notin \lbr j_1,\cdots,j_p\rbr\ni r\\ r\notin\lbr i_1,\cdots,i_p\rbr\ni l\\ \lbr j_1,\cdots,j_p\rbr\setminus\lbr r\rbr=\lbr i_1,\cdots,i_p\rbr\setminus\lbr l\rbr}}\frac{1}{(\la_{j_1}+\cdots+\la_{j_p})( \la_{j_1}+\cdots+\la_{j_p})}\ \ {\rm if}\ \ \substack{k\neq l, r\neq s\\ k=s,\ l=r};\\
			0\ \ {\rm otherwise}.\end{cases}$$
	\end{lemma}
	\begin{proof} The formulas again follow from (\ref{mpandda}) and differentiation of determinants. Recall that coefficients in $\mathcal D_{A}$ are linear with respect to the entries of $A$, which explains the lack of additional terms. It is also interesting
		to note that there doesn't seem to be an easy way to derive Lemma \ref{Hessfla2} through the expression in terms of elementary symmetric polynomials.
	\end{proof}
A consequence of the above formulas is the following result:
\begin{lemma}\label{1kk1}
		Let $A$ be a diagonal matrix with eigenvalues $\la_1\geq\cdots\geq\la_n$. If $\la_1>\la_k$ then 
	\begin{equation*}\frac{\partial^2\widetilde{\mathcal M}_p}{\partial a_{1k}a_{k1}}(A)=\widetilde{\mathcal M}_p^{1k,k1}(A)=\frac{\widetilde{\mathcal M}_p^{kk}-\widetilde{\mathcal M}_p^{11}}{\la_1-\la_k}.
	\end{equation*}
\end{lemma}
\begin{proof}
As $A$ is diagonal we have $\M^{1k}(A)=0$ for $k\neq 1$. Then
$$\frac{\partial^2\widetilde{\mathcal M}_p}{\partial a_{1k}a_{k1}}(A)=\frac1{\binom np}\M(A)^{\frac1{\binom np}-1}\M^{1k,k1}(A)$$
and the formula follows from direct calculations. We also refer to \cite{S} where analogous formula is proved in a much broader context.
\end{proof}
	\subsection{Concavity of $\M^{1/\binom np}$}
	Using Lemmas \ref{Hessfla1} and \ref{Hessfla2} it is a matter of routine calculation to prove that the map $P_p\ni A\rightarrow \widetilde{\mathcal M}_p(A)$ is concave.  Below we present another proof of this fact which is modelled on analogous result for the Monge-Amp\`ere equation- see \cite{Bl96}:
	\begin{lemma}\label{concave}
		Given $A\in P_p$ one has
		$$\widetilde{\mathcal M}_p(A)=\M(A)^{1/\binom np}=inf_{C\in \hat{P}_p}\sum_{k,l=1}^n\M^{kl}(C)A_{kl},$$
		where $\hat{P}_p$ stands for the set of matrices $C\in P_p$, such that $\M(C)=\left[\binom np\right]^{\frac{-\binom np}{\binom np-1 }}.$
	\end{lemma}
	\begin{proof}
		Obviously $\M(A)^{1/\binom np}\geq inf_{C\in \hat{P}_p}\sum_{k,l=1}^n\M^{kl}(C)A_{kl}$ as follows from choosing $C$ as a suitable rescaling of $A$ whenever $\M(A)> 0$. If in turn $\M(A)=0$ then we approximate $A$ by $A+tI$, $I$
		being the identity matrix and $t>0$, and conclude by continuity as $t\searrow 0$.
		
		The reverse inequality is a special case of the G\aa rding inequality- see \cite{Ga}, as $\M$ is a G\aa rding operator- see Section 4 in \cite{HL1}. For an alternative simpler proof we refer to \cite{ADO20}.
	\end{proof}
	Using Lemma \ref{concave} it is straightforward to prove that the map 
	$$P_p\ni A\rightarrow \M(A)^{1/\binom np}$$
	is concave: 
	\begin{corollary}
		$P_p\ni A\rightarrow \widetilde{\mathcal M}_p(A)$ is concave mapping on the set of Hermitian matrices with eigenvalues in $\P$.
	\end{corollary}
	\begin{proof}
		
		\begin{align*}
			&\widetilde{\mathcal M}_p(\frac{A+B}2)=\frac{\M({A+B})^{1/\binom np}}2=inf_{C\in \hat{P}_p}\frac{\sum_{k,l=1}^n\M^{kl}(C)(A_{kl}+B_{kl})}{2}\\
			&\geq inf_{C\in \hat{P}_p}\frac{\sum_{k,l=1}^n\M^{kl}(C)A_{kl}}{2}+inf_{C\in \hat{P}_p}\frac{\sum_{k,l=1}^n\M^{kl}(C)B_{kl}}{2}\\
			&=\frac{\widetilde{\mathcal M}_p(A)+\widetilde{\mathcal M}_p(B)}{2},
		\end{align*}
		which implies the claimed concavity.
	\end{proof}
	\begin{remark}
		Lemma \ref{concave} in fact shows that  $\M$ is a special case of a Bellman operator.
	\end{remark}
	
	\subsection{Further properties of $\mathcal P_p$}
	Below we list some subtler properties of the cones $\P$ and $P_p$ which will be crucial in the establishment of the a priori estimates later on.
	
	We begin with an analogue of Claim (3.10) from \cite{CW}. For diagonal matrices the argument is even easier in the case of the $\M$ operator:
	\begin{lemma}\label{usedingradient}
		Suppose that the diagonal matrix $A=diag (\la_1,\cdots,\la_n)$ belongs to $P_p$ and $\la_1\geq\cdots\geq\la_n$. Suppose moreover that $j\geq n-p+1$. Then there is a constant $\theta=\theta(n,p)$, such that
		$$\M^{jj}(A)\geq \theta \sum_{l=1}^n\M^{ll}(A).$$
	\end{lemma}
	\begin{proof}
		As $\M^{jj}(A)=\M(A)\sum_{j\in \lbr k_1,\cdots,k_p\rbr}\frac{1}{\la_{k_1}+\cdots+\la_{k_p}}$
		it suffices to prove that for {\it any} ordered $p$-tuple $(l_1,\cdots,l_p)$ the term $\frac{1}{\la_{l_1}+\cdots+\la_{l_p}}$ is dominated by one of the terms defining $\M^{jj}(\la)$ i.e. involving a $p$-tuple containing $j$. But by assumption we have $j\geq n-p+1$
		i.e. the $p$-tuple $(n-p+1,\cdots, n)$ contains $j$ and the corresponding sum is clearly the smallest possible.
	\end{proof}
	In the proof of the first order estimate we shall need a version of Lemma \ref{usedingradient} for (special) non-diagonal matrices. Unfortunately the idea to compare any main minor with the minor corresponding to the smallest entry on the diagonal breaks down as the following example shows:
	\begin{example}
		For any sufficiently small $\e>0$ the matrix
		$$\left[\begin{array}{lll}
			1&\e&\e\\
			\e&\frac1\e&\frac1\e-\e\\
			\e&\frac1\e-\e&\frac1\e
		\end{array}\right]
		$$
		is positive definite, the smallest diagonal entry is $1$ and the corresponding $2\times2$ minor is equal to $2-\e^2$ which clearly cannot majorize the sum of all main minors as $\e\searrow 0^{+}$.
	\end{example}
	
	In order to formulate the mentioned  non-diagonal version we need a definition:
	\begin{definition}
		A matrix $A$ is said to be an arrowhead matrix if all off diagonal entries are zero except possibly for the first row and the first column.
	\end{definition}
	
	\begin{lemma}\label{newwritten} Let $A\in P_p$ be an arrowhead matrix.  If additionally $a_{11}\leq -c<0$ for some $c>\frac 2n(\M(A))^{1/n}\binom{n-1}{p}$ then for some $\theta>0$ dependent on $n$,$p$ and $c$  we have
		$$\M^{11}(A)=\frac{\pa \M}{\pa a_{11}}(A)\geq \theta\sum_{l=1}^n\M^{ll}(A).$$ 
	\end{lemma}
	\begin{proof}
		Denote by $\Theta$ the set of all ordered $p$-tuples containing $1$, and by $\Xi$- the remaining ordered $p$-tuples. We denote the entries of $\mathcal D_A$  suggestively by $u_{\alpha\beta}$. Recall that $u_{\alpha\alpha}=\sum_{i\in\alpha}a_{ii}$.
		Observe that (compare Examples \ref{a} and \ref{b}) for an arrowhead matrix $A$ the matrix $\mathcal D_A$ can be expressed as
		\begin{equation}\label{da}
			\left[\begin{array}{l|l}
				\left(diag(u_{\theta\theta})\right)_{\theta\in \Theta}&\left(u_{\theta\xi}\right)_{\theta\in\Theta,\xi\in\Xi}\\
				\hline
				\left(u_{\xi\theta}\right)_{\theta\in\Theta,\xi\in\Xi}&\left(diag(u_{\xi\xi})\right)_{\xi\in \Xi}
			\end{array}
			\right],
		\end{equation}
		i.e. $(u)_{\alpha\beta}$ is a block matrix with upper left and lower right blocks being diagonal matrices. This follows since $A$ being arrowhead implies that nontrivial input in $\mathcal D_A$ can occur either from the action of $A$ on the first basis vector $e_1$ or from the $e_1$-component of $Ae_k$ or finally from the $e_k$-component of $Ae_k$ for some $k\geq1$. Another property which will be crucial in the sequel is that for fixed $\beta\in\Xi$ the nonzero entries in the column $\beta$ could appear only on the rows $\alpha\in \Theta_\beta$  which  are given by
		$$\Theta_\beta:=\lbr \alpha\in\Theta|\ \exists q\in\beta, \alpha=\lbr{1}\rbr\cup\beta\setminus\lbr q\rbr\rbr.$$
		
		At this moment we wish to emphasize that in general it is not true that $\forall \alpha\in\Theta_\beta\ u_{\alpha\alpha}\leq u_{\beta\beta}$, as we do not assume that $a_{11}$ is the smallest diagonal entry of $A$. On the other hand one has
		$$(p-1)u_{\beta\beta}+pa_{11}=\sum_{\alpha\in\Theta_\beta}u_{\alpha\alpha}$$
		and hence
		\begin{equation}\label{bbeta}
			\forall\beta\in\Xi\ \forall \alpha\in\Theta_\beta\ \  u_{\beta\beta}\geq max\lbr\frac1{p-1}u_{\alpha\alpha}, \frac{pc}{p-1}\rbr.
		\end{equation}

		As $\mathcal D_A$ is positive definite the determinant of any $(\binom{n-1}{p-1}+1)\times(\binom{n-1}{p-1}+1)$ main minor matrix is positive. We apply this to any minor formed by the first $\binom{n-1}{p-1}$ rows and columns (the $\Theta$-block) to which we add the row and column $\beta$.
		Then we obtain
		\begin{equation}\label{beal}
			u_{\beta\beta}\geq \sum_{\alpha\in\Theta_\beta}\frac{|u_{\alpha\beta}|^2}{u_{\alpha\alpha}}.
		\end{equation}
		
		
		Fix now a small $\e>0$ to be chosen later on.
		
		Expanding $det(\mathcal D_A)$ with respect to the column $\beta$ we obtain
		
		$$\M(A)=u_{\beta\beta}U^{\beta\beta}+\sum_{\alpha\in\Theta_\beta}u_{\alpha\beta}U^{\alpha\beta},$$
		where $U^{\alpha\beta}$ denotes the co-factor matrix of $(u)_{\alpha\beta}$. Of course $(U^{\alpha\beta})_{\alpha\beta}$ is positive definite as $A\in P_p$ and thus $|U^{\alpha\beta}|^2\leq U^{\alpha\alpha}U^{\beta\beta}$. Hence
		\begin{align*}
			&\M(A)\geq u_{\beta\beta}U^{\beta\beta}-\sum_{\alpha\in\Theta_\beta}|u_{\alpha\beta}|\sqrt{U^{\alpha\alpha}U^{\beta\beta}}\\
			&\geq u_{\beta\beta}U^{\beta\beta}-\sqrt{\left(\sum_{\alpha\in\Theta_\beta}\frac{|u_{\alpha\beta}|^2}{u_{\alpha\alpha}}\right)\left({\sum_{\alpha\in\Theta_\beta}U^{\alpha\alpha}U^{\beta\beta}u_{\alpha\alpha}}\right)}\\
			&\geq u_{\beta\beta}U^{\beta\beta}-\sqrt{u_{\beta\beta}U^{\beta\beta}\sum_{\alpha\in\Theta_\beta}U^{\alpha\alpha}u_{\alpha\alpha}}\\
			&\geq (1-\e)u_{\beta\beta}U^{\beta\beta}-\frac1{4\e}\sum_{\alpha\in\Theta_\beta}U^{\alpha\alpha}u_{\alpha\alpha},
		\end{align*}
		where we have used Cauchy-Schwarz inequality, then (\ref{beal}) and finally the elementary inequality $xy\leq \e x^2+\frac1{4\e}y^2$. 
		
		Thus exploiting (\ref{bbeta}) we obtain
		\begin{equation}\label{123}
			\frac{(p-1)}{4\e}u_{\beta\beta} \sum_{\alpha\in\Theta_\beta}U^{\alpha\alpha}+\M(A)\geq\frac1{4\e}\sum_{\alpha\in\Theta_\beta}U^{\alpha\alpha}u_{\alpha\alpha}+ \M(A)\geq (1-\e)u_{\beta\beta}U^{\beta\beta}.
		\end{equation}
		Dividing by $u_{\beta\beta}$, applying (\ref{bbeta}) and then summing over $\beta\in \Xi$ results in
		\begin{equation*}
			C(n,p,\e)\sum_{\alpha\in\Theta}U^{\alpha\alpha}+\binom{n-1}{p}\frac{\M(A)}{c}\geq \sum_{\beta\in\Xi}(1-\e)U^{\beta\beta}+\frac12\sum_{\alpha\in\Theta}U^{\alpha\alpha}.
		\end{equation*}
		Note that 
		$$\frac 12 \sum_{\beta\in\Xi}U^{\beta\beta}+\frac12\sum_{\alpha\in\Theta}U^{\alpha\alpha}=\frac12\sigma_{n-1}(\mathcal D_A)$$
		$$\geq \frac n2 (\sigma_n(\mathcal D_A))^{(n-1)/n}$$
		by Maclaurin inequality. But by our assumption on $c$ the latter quantity satisfies the bound
		
		$$\frac n2 (\sigma_n(\mathcal D_A))^{(n-1)/n}=\frac n2(\M(A))^{(n-1)/n}\geq \binom{n-1}{p}\frac{\M(A)}{c}.$$
		
		Thus fixing $\e=\frac 14$, say, we finally obtain
		\begin{equation}\label{12345}
			C(n,p,\e)\M^{11}(A)=C(n,p,\e)\sum_{\alpha\in\Theta}U^{\alpha\alpha}\geq \sum_{\beta\in\Xi}U^{\beta\beta},
		\end{equation}
		which yields the claimed result.

	\end{proof}
	\begin{remark}
		It is very likely that through a more careful analysis one could remove the dependency of $c$ on $\M(A)$- this is easily seen to be true if $p=n-1$. We have not pursued this as the stated version is satisfactory for the applications.
	\end{remark}
	
	\begin{lemma}\label{toadd}
		Suppose that the vector $\la=(\la_1,\cdots,\la_n)$ belongs to $\P$ and $\la_1\geq\cdots\geq\la_n$. Then
		$$\widetilde{\mathcal M}_p^{11}(\la)\la_{1}\geq \frac{1}{n}(\M(\la))^{\frac1{\binom np}}.$$
	\end{lemma}
	\begin{proof}
		Recall that $\widetilde{\mathcal M}_p^{11}\la_1$ is simply
		$$\frac{1}{\binom np}(\M(\la))^{\frac1{\binom np}}\sum _{1\in \lbr k_1,\cdots,k_p\rbr}\frac{\la_1}{\la_{k_1}+\cdots+\la_{k_p}}$$
		and the proof follows form the trivial inequality 
		$$\la_{k_1}+\cdots+\la_{k_p}\leq p\la_1.$$
	\end{proof}
	
	For the second order interior estimate one has to exploit the concavity of $\widetilde{\mathcal M}_p$ in order to handle the third order terms.
	Next lemma establishes the relevant inequality in special cases analogously to its $\sigma_k$ counterpart (compare with Claim (4.13) in \cite{CW} and its proof):
	\begin{lemma}\label{usedinhessian}
		Suppose that the vector $\la=(\la_1,\cdots,\la_n)$ belongs to $\P$ and $\la_1\geq\cdots\geq\la_n$. Suppose $p=2,n-2$ or $n-1$. Then, given any sufficiently small $\delta>0$ there is an $\e=\e(p,n,\delta)>0$ with the following property: if $\e\la_1\geq \la_{n-p+1}$, then for every $i=2,3,\cdots, n$
		$$-\M^{1i,i1}(\la)\geq(\frac12+\delta)\M^{ii}(\la)\frac1{\la_1}.$$
	\end{lemma}
	\begin{proof}
		Let us handle the $p=2$ case first.
		We know from Lemmas \ref{Hessfla1} and \ref{Hessfla2} that
		$$\M^{1j,j1}(\la)$$
		$$=-\M(\la)\sum_{j\neq 1,i}\frac{1}{(\la_{1}+\la_{j})(\la_{i}+\la_{j})},$$
		while
		$$\M^{jj}(\la)=\M(\la)\sum_{j\neq i}\frac{1}{\la_{i}+\la_j}.$$
		Hence we need to prove that under the assumptions as in the lemma
		\begin{equation}\label{toprove2}
		\sum_{j\neq 1,i}\frac{1}{(\la_{1}+\la_{j})(\la_{i}+\la_{j})}
		\end{equation}
		$$\geq (\frac12+\delta) \sum_{j\neq i}\frac{1}{\la_1(\la_{i}+\la_j)}=(\frac12+\delta) [\frac{1}{\la_1(\la_1+\la_i)}+\sum_{j\neq 1,i}\frac{1}{\la_1(\la_{i}+\la_j)}].$$
		
		Suppose that $i<n-1$ (this forces $n\geq 4$). By assumption we have
		$$\frac{1}{(\la_{1}+\la_{n-1})(\la_{i}+\la_{n-1})}+\frac{1}{(\la_{1}+\la_{n})(\la_{i}+\la_{n})}\geq$$ $$\frac1{(1+\e)\la_1}[\frac{1}{\la_{i}+\la_{n-1}}+\frac{1}{\la_{i}+\la_{n}}]\geq \frac{(1-2\e)}{\la_1}[\frac{1}{\la_{i}+\la_{n-1}}+\frac{1}{\la_{i}+\la_{n}}]$$ 
		if $\e$ is small enough.
		
		On the other hand recall that $\la_1\geq\la_j$ and $\la_j\geq \la_{n-1}\geq\la_n$ for $j<n-1$. Assuming $1/2-\e-\delta>0$ we have
		$$\sum_{j\neq 1,i}\frac{1}{(\la_{1}+\la_{j})(\la_{i}+\la_{j})}=\sum_{j\neq 1,i; j<n-1}\frac{1}{(\la_{1}+\la_{j})(\la_{i}+\la_{j})}+$$
		$$\frac{1}{(\la_{1}+\la_{n-1})(\la_{i}+\la_{n-1})}+\frac{1}{(\la_{1}+\la_{n})(\la_{i}+\la_{n})}\geq$$
		$$\sum_{j\neq 1,i; j<n-1}\frac{1}{2\la_{1}(\la_{i}+\la_{j})}+\frac{(1/2-2\e-\delta)}{\la_1}[\frac{1}{\la_{i}+\la_{n-1}}+\frac{1}{\la_{i}+\la_{n}}]+$$
		$$(\frac12+\delta)\frac1{\la_1}[\frac{1}{\la_{i}+\la_{n-1}}+\frac{1}{\la_{i}+\la_{n}}].$$
		We now divide $$(2/3+1/3)(\frac12-2\e-\delta)\frac1{\la_1}[\frac{1}{\la_{i}+\la_{n-1}}+\frac{1}{\la_{i}+\la_{n}}]\geq\frac43(\frac12-2\e-\delta)\frac1{\la_1}\frac{1}{\la_{1}+\la_{i}}+$$
		$$\frac23\frac{(1/2-2\e-\delta)}{n-3}\sum_{j\neq i;j<n-1}\frac1{\la_{1}(\la_{i}+\la_{j})}.$$
		
		Coupling all the estimates we end up with
		$$\sum_{j\neq 1,i}\frac{1}{(\la_{1}+\la_{j})(\la_{i}+\la_{j})}\geq$$
		$$ \sum_{j\neq i; j<n-1}(\frac{1}{2}+\frac23\frac{1/2-\e-\delta}{n-3})\frac1{\la_{1}(\la_{i}+\la_{j})}+$$
		$$(\frac12+\delta)\frac1{\la_1}[\frac{1}{\la_{i}+\la_{n-1}}+\frac{1}{\la_{i}+\la_{n}}]+\frac43(\frac12-2\e-\delta)\frac1{\la_1}\frac{1}{\la_{1}+\la_{i}}.$$
		
	Obviously if $\delta$ is small enough $\frac23\frac{1/2-\e-\delta}{n-3}\geq\delta, \frac43(\frac12-2\e-\delta)\geq \frac12+\delta$ and the claim is established.
		
		If, in turn, $i=n-1$ or $n$ then one of the summands of $\sum_{j\neq 1,i}\frac{1}{(\la_{1}+\la_{j})(\la_{i}+\la_{j})}$ is 
		$$\frac{1}{(\la_{1}+\la_{\lbr n-1,n\rbr\setminus\lbr  i\rbr})(\la_{i}+\la_{\lbr n-1,n\rbr\setminus\lbr  i\rbr})}=\frac{1}{(\la_{1}+\la_{\lbr n-1,n\rbr\setminus\lbr  i\rbr})(\la_{n-1}+\la_{n})}\geq$$
		$$\frac 1{6\e}\frac{1}{\la_{1}(\la_{1}+\la_{\lbr n-1,n\rbr\setminus\lbr  i\rbr})}+\frac2{3(1+\e)\la_1}\frac1{\la_{n-1}+\la_n}\geq$$
		$$\frac {(1-\e)}{6\e(1+\e)}\frac{1}{\la_{1}(\la_{1}+\la_i)}+\frac2{3(1+\e)\la_1}\frac1{(\la_{n-1}+\la_n)}$$
		as $\la_n\geq-\e\la_1$ (the factor $\frac{1-\e}{1+\e}$ is only needed in the $i=n$ case to estimate $\la_1+\la_{n-1}$ by $\la_1+\la_n$). For small $\delta>0$ dependent on $n$ and even smaller $\e>0$ these two terms absorb
		
	$$(\frac12+\delta)[\frac{1}{\la_1(\la_{n-1}+\la_n)}+\frac{1}{\la_1(\la_i+\la_1)}]+\delta\sum_{j\neq 1;j<n-1}\frac1{\la_1(\la_i+\la_j)}$$
	exploiting now the fact that $\la_{n-1}+\la_n\leq \la_i+\la_j$.
	
		It remains to observe that
		$$\sum_{j\neq 1,i; j<n-1}\frac{1}{(\la_{1}+\la_{j})(\la_{i}+\la_{j})}\geq\frac12\sum_{j\neq 1,i; j<n-1}\frac1{\la_1(\la_i+\la_j)}$$
		which finishes the proof for the $p=2$ case.

	Consider now $p=n-2$. In this case all the eigenvalues, except possibly for $\la_1$ and $\la_2$ are small when compared with $\la_1$.
	
	The inequality to be proven is equivalent to
	$$\sum_{\lbr j_1,\cdots,j_{n-3}\rbr\subset\lbr 2,3\cdots,n\rbr\setminus \lbr i\rbr}\frac{1}{(\la_1+\la_{j_1}+\cdots+\la_{j_{n-3}})(\la_i+\la_{j_1}+\cdots+\la_{j_{n-3}})}\geq$$
	$$(\frac12+\delta)\frac1{\la_1}[\sum_{\lbr j_1,\cdots,j_{n-3}\rbr\subset\lbr 2,3,\cdots,n\rbr\setminus \lbr i\rbr}\frac1{\la_i+\la_{j_1}+\cdots+\la_{j_{n-3}}}+$$
	$$\sum_{\lbr s_1,\cdots,s_{n-4}\rbr\subset\lbr 2,3\cdots,n\rbr\setminus \lbr i\rbr}\frac1{\la_1+\la_i+\la_{s_1}+\cdots+\la_{s_{n-4}}}].$$
	
	Consider first the case $i=2$. Note that then all the $\la_{j_k}$'s and $\la_{s_r}$'s are small when copared with $\la_1$ if $\e$ is chosen small enough. Hence, just like in the $p=2$ case we estimate
	$$\frac{1}{(\la_1+\la_{j_1}+\cdots+\la_{j_{n-3}})(\la_2+\la_{j_1}+\cdots+\la_{j_{n-3}})}\geq$$
	$$(\frac35+\frac25)\frac{1-(n-2)\e}{\la_1(\la_2+\la_{j_1}+\cdots+\la_{j_{n-3}})}.$$
	The $\frac35$-part dominates $(\frac12+\delta)\frac1{\la_1}\sum_{\lbr j_1,\cdots,j_{n-3}\rbr\subset\lbr 2,3\cdots,n\rbr\setminus \lbr i\rbr}\frac1{\la_2+\la_{j_1}+\cdots+\la_{j_{n-3}}}$ if $\delta$ is small enough.
	
	In order to handle $\sum_{\lbr s_1,\cdots,s_{n-4}\rbr\subset\lbr 2,3\cdots,n\rbr\setminus \lbr i\rbr}\frac1{\la_1+\la_i+\la_{s_1}+\cdots+\la_{s_{n-4}}}$ (such terms exist only if $n\geq 4$) note that
	
	$$\frac25\frac{1-(n-2)\e}{\la_1(\la_2+\la_{j_1}+\cdots+\la_{j_{n-3}})}$$
	$$\geq\sum_{r=1}^{n-3}\frac35\frac{(1-(n-2)\e)}{(n-3)\la_1(\la_1+\la_2+\la_{j_1}+\cdots+\la_{j_{n-3}}-\la_{j_r})}.$$
	Indeed, this follows from
	$$\frac1{\la_2+\la_{j_1}+\cdots+\la_{j_{n-3}}}\geq \frac 32\frac{1}{\la_1+\la_2+\la_{j_1}+\cdots+\la_{j_{n-3}}-\la_{j_r}}$$
	which holds as $\la_2\leq\la_1$ and $\e$ is small enough.
	
	Summing over all $(n-3)$ tuples $(j_1,\cdots,j_{n-3})$ and rearranging we obtain
	$$\frac25\sum_{\lbr j_1,\cdots,j_{n-3}\rbr\subset\lbr 3\cdots,n\rbr}\frac{1-(n-2)\e}{\la_1(\la_2+\la_{j_1}+\cdots+\la_{j_{n-3}})}\geq$$
	$$\frac35\sum_{\lbr s_1,\cdots,s_{n-4}\rbr\subset\lbr 3,\cdots,n\rbr}\frac{1-(n-2)\e}{\la_1(\la_1+\la_2+\la_{s_1}+\cdots+\la_{s_{n-4}})}$$
	which again dominates the remaining terms on the right hand side provided $\delta$ is small enough.
	
	Next we focus on the case $i\geq 3$. Then each term
	$$\frac{1}{(\la_1+\la_{j_1}+\cdots+\la_{j_{n-3}})(\la_i+\la_{j_1}+\cdots+\la_{j_{n-3}})}$$
	absorbs
	$$(\frac12-(n-3)\e)\frac{1}{\la_1(\la_i+\la_{j_1}+\cdots+\la_{j_{n-3}})}$$
		if $2\in\lbr j_1,\cdots,j_{n-3}\rbr$ and
	$$(1-(n-2)\e)\frac{1}{\la_1(\la_i+\la_{j_1}+\cdots+\la_{j_{n-3}})}$$
		otherwise.
		
		Hence it suffices to establish the bound
		$$(\frac12-(n-2)\e-\delta)\frac{1}{\la_1(\la_3+\la_{4}+\cdots+\la_{n})}\geq$$
			$$(\delta+(n-3)\e)\sum_{{1{\rm or}2\in\lbr j_1,\cdots,j_{n-3}\rbr}}\frac{1}{\la_1(\la_i+\la_{j_1}+\cdots+\la_{j_{n-3}})}.$$	
		
		Observe however that 
		$$\frac{1}{\la_1(\la_3+\la_{4}+\cdots+\la_{n})}\geq\frac{1}{\la_1(\la_i+\la_{j_1}+\cdots+\la_{j_{n-3}})}$$
		for each individual such tuple $\lbr j_1,\cdots,j_{n-3}\rbr$. Thus the left hand side will be greater than a constant (dependent on $n$) times the whole sum and hence the inequality is justified when $\delta$ is small enough.
		
		Finally the case $p=n-1$ is analogous to $p=n-2$ axcept that there is no subcase when $i$ is small.	
		
	\end{proof}
We remark that the constant in the inequality being larger than $\frac12$ is crucial for the second order interior estimates when following the proof from \cite{CW}. Unfortunately in general it is not true that $-\M^{1i,i1}(\la)$ dominates $\M^{ii}(\la)\frac1{\la_1}$ times a constant larger than $\frac 12$ as the following example shows:
\begin{example}
Take $n=9,\ p=3$. Fix $\la_1=\la_2=1,\cdots,\la_6=1$ and $\la_7=\la_8=\la_9=\e$ with $\e$ as small as we please. Then
$$-\M^{12,21} (\la)=\frac23+\frac{12}{(2+\e)^2}+\frac3{(1+2\e)^3},$$
while

$$\M^{22}(\la)\frac1{\la_1}=\frac{10}3+\frac{15}{2+\e}+\frac3{1+2\e}.$$
Their ratio as $\e\rightarrow 0^+$ tends to $\frac{40}{83}$ which is less than $\frac12$.
\end{example}

This lack of sufficient concavity in general forces  the application of subtler tools for the second order estimates in general. The next lemma will be crucial. I learned this argument from Professors W. Dong and J. Chu.
\begin{lemma}\label{cruciallemma}
		Suppose that the vector $\la=(\la_1,\cdots,\la_n)$ belongs to $\P$ and $\la_1\geq\cdots\geq\la_n$. Then, given any $\delta>0$ there is an $\e=\e(p,n,\delta)>0$ with the following property: if $\e\la_1\geq \la_{n-p+1}$ and $\la_1\geq 2+\frac2\delta$, then for every $i=2,3,\cdots, n$
	$$-2\frac{\M^{1i,i1}(\la)}{\la_1}+2\frac{\M^{11}(\la)}{\la_1(\la_1-\la_i+1)}\geq(\frac{\delta+2}{\delta+1})\M^{ii}(\la)\frac1{\la_1^2}.$$
	\end{lemma}
\begin{proof}
If $\la_i=\la_1$ then the inequality follows from $\M^{1i,i1}\leq 0, \M^{ii}=\M^{11}$, $\la_1\geq1$ and $2\geq\frac{2+\delta}{1+\delta}$. Assume now $\la_i<\la_1$. Then by Lemma \ref{1kk1} we have
$$-2\frac{\M^{1i,i1}(\la)}{\la_1}+2\frac{\M^{11}(\la)}{\la_1(\la_1-\la_i+1)}=2\frac{\M^{ii}(\la)-\M^{11}(\la)}{\la_1(\la_1-\la_i)}+2\frac{\M^{11}(\la)}{\la_1(\la_1-\la_i+1)}\geq$$
$$2\frac{\M^{ii}(\la)-\M^{11}(\la)}{\la_1(\la_1-\la_i+1)}+2\frac{\M^{11}(\la)}{\la_1(\la_1-\la_i+1)}=\frac{2\M^{ii}(\la)}{\la_1(\la_1-\la_i+1)}\geq\frac{2\M^{ii}(\la)}{\la_1((1+(n-1)\e)\la_1+1)}.$$
It suffices then to prove that
$$\frac{2}{\la_1((1+(n-1)\e)\la_1+1)}\geq (\frac{\delta+2}{\delta+1})\frac1{\la_1^2} $$

Rearranging terms leads to the equivalent inequality
$$[2-(1+(n-1)\e)(\frac{\delta+2}{\delta+1})]\la_1\geq\frac{\delta+2}{\delta+1}.$$
As $\la_1$ is by definition larger than $2+\frac2\delta$ the inequality holds and is strict when $\e$ is replaced by zero and hence remains true if $\e$ is taken sufficiently small.
\end{proof}	
	\section{First order estimate}
	The interior gradient estimate for the $\sigma_k$ equations has been independently proven by \cite{Tr} and \cite{CW}.
	
	In this section we apply the methods of Chou and Wang from \cite{CW} to obtain the following interior gradient estimate analogous to  Theorem 3.2 from that paper:
	\begin{theorem}
		Let $u$ be $p$-plurisubharmonic function in the ball $B_r(x_0)$. Assume that $u\in C^3(B_r(x_0))\cap C^1(\overline{B_r(x_0)})$ and $u$ solves the equation
		$$\M(u)=f(x,u),$$
		for a given non negative Lipschitz function $f$.\footnote{More precisely we assume that $|f|\leq C$, $|\frac{\pa f}{\pa u}|+\sum_j|\frac{\pa f}{\pa x_j}|\leq C$ for some constant $C$ called the Lipschitz bound of $f$. } Then 
		\begin{equation}\label{gradbound}
			|Du(x_0)|\leq C
		\end{equation}
		for some constant $C$ dependent on $n, p, r$,  $sup_{B_r(x_0)}|u|$ and the Lipschitz bound on $f$.
	\end{theorem}
	\begin{proof}
		Below we essentially repeat the argument from \cite{CW} with minor adjustments.
		
		Assume without loss of generality that $x_0=0$. Scaling if necessary one may also assume that $r=1$ and thus we will work throughout in the unit ball $B=B_1(0)$. 
		Define the function
		$$H: \overline{B}\times\mathbb S^n\ni(y,\xi)\rightarrow u_{\xi}(y)(1-|y|^2)\varphi(u(y)),$$
		with $\varphi(t):=\frac{1}{\sqrt{N-t}},\ N:=4sup_{B}|u|$. Suppose that $H$ attains its maximum at $(x,\xi_0)$. Without loss of generality we may assume that $\xi(0)=(1,0,\cdots,0)$ and then the function
		$$\hat{H}(y):=u_1(y)(1-|y|^2)\varphi(u(y))$$
		also has a maximum at $x$. Note that this forces $u_k(x)=0$ for any $k\in\lbr2,\cdots,n\rbr$ as $\frac{\pa}{\pa x_1}$ has to coincide with the gradient direction of $u$ at $x$.
		
		Rotating all but the first coordinates if necessary (see \cite{Tr}) we may assume that at $x$ the Hessian of $u$ is an arrowhead matrix.
		
		We may further assume that $\hat{H}(x)$ is so large so that
		\begin{equation}\label{x00}
			u_1(x)(1-|x|^2)\geq 10N
		\end{equation}
		for otherwise there is nothing to prove.
		
		Of course $D\hat{H}(x)=0$, which results in
		\begin{equation}\label{gradk}
			\forall k\in\lbr1,\cdots,n\rbr\ \ \frac{u_{1k}}{u_1}-\frac{2x_k}{1-|x|^2}+\frac{\varphi'}{\varphi}u_k=0
		\end{equation}
		at $x$.
		
		Furthermore
		\begin{equation}\label{estsforgrad}
			0\geq\sum_{k,l=1}^n\M^{kl}\frac{\hat{H}_{kl}}{\hat{H}}=\sum_{k,l=1}^n\M^{kl}\left[\frac{u_{1kl}}{u_1}-\frac{u_{1k}u_{1l}}{u_1^2}\right]
		\end{equation}
		$$-\sum_{k,l=1}^n\M^{kl}\left[\frac{2\delta_{kl}}{1-|x|^2}+\frac{4x_kx_l}{(1-|x|^2)^2}\right]+\sum_{k,l=1}^n\M^{kl}\left[\frac{\varphi'}{\varphi}u_{kl}+[\frac{\varphi''}{\varphi}-(\frac{\varphi'}{\varphi})^2]u_ku_l\right]$$
		$$:=I+II+III.$$
		
		Recall now that Lemma \ref{basicM} implies that
		$$\sum_{k,l=1}^n\M^{kl}u_{kl}=\binom np f,\ \sum_{k,l=1}^n\M^{kl}u_{kl1}=\pa_1f.$$
		
		Also by (\ref{gradk}) we may exchange $\frac{u_{1k}}{u_1}$ and $\frac{u_{1l}}{u_1}$ in the term $I$ by $\frac{2x_k}{1-|x|^2}-\frac{\varphi'u_k}{\varphi}$ and $\frac{2x_l}{1-|x|^2}-\frac{\varphi'u_l}{\varphi}$, respectively. Thus
		$$I+II\geq -\frac{|\pa_1f|}{u_1}$$
		$$-\sum_{k,l=1}^n\M^{kl}\left[\frac{2\delta_{kl}}{1-|x|^2}+\frac{8x_kx_l}{(1-|x|^2)^2}-2\frac{(x_ku_l+x_lu_k)\varphi'}{(1-|x|^2)\varphi}+u_ku_l(\frac{\varphi'}{\varphi})^2\right].$$
		
		By our choice $\varphi''-2\frac{\varphi'^2}{\varphi}\geq\frac{N^{-5/2}}{16}$. If ${M}:=\sum_{k=1}^n\M^{kk}$ at $x$ we obtain
		\begin{equation}
			0\geq I+II+III\geq \frac{\varphi'}{\varphi}\binom np f-\frac{|\pa_1f|}{u_1}-C{M}(\frac{1}{(1-|x|^2)^2}+\frac{u_1}{(1-|x|^2)\varphi})
		\end{equation}
		$$+\frac{N^{-5/2}}{16}\sum_{k,l=1}^n\M^{kl}u_ku_l.$$
		
		Exploiting the non-negativity of $f$ together with the positive-definiteness of $\M^{kl}$ we get after mulitplying by $16N^{5/2}\hat{H}$ the following inequality
		\begin{equation}\label{finalgrad}
			0\geq -\varphi(1-|x|^2)16N^{5/2}|\pa_1f|+(1-|x|^2)\varphi\M^{11}u_1^3-C{M}(\frac{N^2}{(1-|x|^2)}+Nu_1^2).
		\end{equation}
		
		As $f$ was assumed to be Lipschitz the first term is bounded from below by $-C\hat{H}$.
		
		Recall that from  (\ref{x00}) we have $u_1(x)(1-|x|^2)\geq 10N$. Thus (\ref{gradk}) for $k=1$ reads
		$$\frac{u_{11}}{u_1}=\frac{2x_k\varphi-\varphi'(1-|x|^2)u_1}{(1-|x|^2)\varphi}$$
		$$\leq \frac{4(N-u)\varphi'-\varphi'(1-|x|^2)u_1}{(1-|x|^2)\varphi}\leq \varphi'\frac{5N-(1-|x|^2)u_1}{(1-|x|^2)\varphi}$$
		$$\leq\frac{-\varphi'u_1}{2\varphi}<0.$$
		
		Hence 
		$$u_{11}\leq \frac{-\varphi'u_1^2}{2\varphi}\leq -c<0$$
		
		Note that we can safely assume that $c>\frac 2n(\M(A))^{1/n}\binom{n-1}{p}$ for otherwise $\hat{H}$ is bounded at $x$.
		
		Thus we can apply Lemma \ref{newwritten} (recall that we have assumed that $D^2u$ is arrowhead at $x$). 
		As a result we get that $\M^{11}\geq \theta M$ and hence (\ref{finalgrad}) reduces to
		
		$$0\geq -C(1-|x|^2)u_1+CM\left((1-|x|^2)u_1^3-C\frac{N^2}{(1-|x|^2)}-CNu_1^2\right)$$
		and thus finally $(1-|x|^2)u_1\leq CN.$
		
		Now the claimed result follows from evaluating $\hat{H}$ at $x_0$.
	\end{proof}
	\begin{remark}
		Similarly to \cite{CW} the estimate holds for slightly more general right hand sides. Also in the case of a constant $f$ a careful examination of the argument above reveals that $|Du(x_0)|\leq \frac{CN}{r}$ for some $C$ dependent only on $n$ and $p$. 
	\end{remark}
	An immediate application is the following analogue of Corollary 4.1 from \cite{Wa2}:
	\begin{corollary}
		Let $u\in C^3(\mathbb R^n)$ be an entire $p$-plurisubharmonic solution to
		$$\M(u)=0.$$
		If $u$ is bounded, or more generally if $u=o(|x|)$ for large $x$, then $u$ is a constant.
	\end{corollary}

	\section{Second order estimate}
	Second order interior estimates for fully nonlinear elliptic PDEs as a rule require much subtler techniques. Instead of working in a ball we shall assume and utilize the existence of an upper barrier function $w$. Such a barrier was considered in \cite{TU}, where
	$w$ was taken as a solution to a homogeneous equation with
	suitable boundary conditions.

	Below we prove the following theorem which is an analogue of Theorem 4.1 from \cite{CW}:
	\begin{theorem}\label{hessbound}
		Let $\Om$ be a bounded domain in $\mathbb R^n$ and  the $p$-plurisubharmonic function $u\in C^4(\Omega)\cap C^2(\overline{\Om})$ solves the equation
		$$\M(u)=f(x,u),$$
		where $f\in C^{1,1}(\overline{\Om}\times\mathbb R)$ is given. Suppose that $f\geq f_0>0$, and that there is a $p$-plurisubharmonic function $w$, satisfying $w>u$ in $\Om$ with equality on $\partial\Om$. Then for any fixed $\delta>0$ there is a bound
		\begin{equation}\label{c11}
			(w-u)^{1+\delta}|D^2u(x)|\leq C,
		\end{equation}
		where $C$ depends on $n,p,f_0,sup_{\Om}(|Du|+|Dw|)$ and $||f||_{C^{1,1}}$.
	\end{theorem}
	\begin{proof}
		As we mentioned beforehand we shall use the casewise approach similar to \cite{CW}. In fact the  major difference is the usage of Lemma \ref{cruciallemma}. We provide the details for the sake of completeness.
		
		Throughout the argument we shall work with a fixed $\delta>0$. Note that if $\delta_1<\delta_2$ and the result holds true for $\delta_1$ then it also easily holds for $\delta_2$. Hence it suffices to prove the Theorem only
		for $\delta$'s sufficiently small, say for $\delta\leq\frac12$, so that $1-\delta-\delta^2\geq \frac14$. 
		
		We shall work with the largest eigenvalue $\la_1(x)$ of the Hessian of $u(x)$. As $\la_1$ may fail to be smooth at points where $\la_1=\la_2$ we apply the (nowadays standard) perturbation trick. Rotating the coordinates if necessary we may assume that at a given point $x_0\in\Omega$  $\la_1$ is the eigenvalue with respect to the eigenvector $\frac{\partial}{\partial x_1}$. Then we consider the constant coefficient matrix
		$$B_{ij}=\delta_{ij}-\delta_{i1}\delta_{j1}$$
		and let $\hat{\la}_i,\ i=1,\cdots,n$ denote the eigenvalues of
		$u_{ij}-B_{ij}$ ordered in a decreasing order. As $B_{ij}$ is non negative definite we have  
		$$\hat{\la}_1(x)\leq \la_1(x)$$
		with equality at $x_0$, while
		$$\forall j>1\ \hat{\la}_j(x_0)\leq \la_2(x_0)-1<\hat{\la}_1(x_0),$$
		which implies that $\hat{\la}_1$ is a smooth function near $x_0$. 
		
		Below we shall use the following classical formulas (see \cite{S} for a proof)
		\begin{equation}\label{Spruck1}
	\frac{\partial \hat{\la}_1}{\partial x_k}(x_0)=\frac{\partial u_{11}}{\partial x_k}(x_0);
		\end{equation} 
	\begin{equation}\label{Spruck2}
		\frac{\partial^2 \hat{\la}_1}{\partial x_k\partial x_l}(x_0)=\frac{\partial^2 u_{11}}{\partial x_k\partial x_l}(x_0)+2\sum_{m>1}\frac{u_{1mk}u_{1ml}}{\la_1-\la_m+1}.
	\end{equation} 
		
		Define the function
		\begin{equation}\label{g}
			G:{\Omega}\ni x\rightarrow (w(x)-u(x))^{1+\delta}\varphi(\frac12|Du(x)|^2)\la_1(x),
		\end{equation}
		where $\varphi$ is given by 
		\begin{equation}\label{whattochoose}
			\varphi(t):=(\frac1{2S-t})^{\delta+\delta^2}, 
		\end{equation}
		$S$ being equal to $sup_{\Om}|Du|^2$. Observe that $\varphi$ is convex, uniformly bounded from above and below and satisfies
		\begin{equation}\label{phi}
			\frac{\varphi''}{\varphi}=\frac{1+\delta+\delta^2}{\delta+\delta^2}(\frac{\varphi'}{\varphi})^2.
		\end{equation}

		We assume that the maximum of $G$ occurs at $x_0$. Rotating the coordinates if necessary we may assume that $D^2u$ is a diagonal matrix and the direction of the eigenvector corresponding to the largest eigenvalue $\la_1$ is $(1,0,\cdots,0)$. Then in the new
		coordinates the function 
		$$\hat{G}(x):=(w(x)-u(x))^{1+\delta}\varphi(\frac12|Du(x)|^2)\hat{\la}_1(x)$$
		also has a maximum at $x_0$. As explained earlier $\hat{G}$ is twice differentiable at $x_0$, so we compute exploiting (\ref{Spruck1}) and (\ref{Spruck2})
		\begin{equation}\label{21}
			\forall k\in\lbr1,\cdots,n\rbr\ \ 0=\frac{(\hat{G})_k}{\hat{G}}=(1+\delta)\frac{w_k-u_k}{w-u}+\frac{\varphi'}{\varphi}\sum_{l=1}^nu_{lk}u_l+\frac{u_{11k}}{u_{11}}
		\end{equation}
		and
		\begin{equation}\label{22}
			0\geq\sum_{k=1}^n \widetilde{M}_{p}^{kk}\frac{\hat{G}_{kk}}{\hat{G}}=\sum_{k=1}^n\widetilde{M}_{p}^{kk}(1+\delta)\left[\frac{w_{kk}-u_{kk}}{w-u}-\frac{(w_k-u_k)^2}{(w-u)^2}\right]
		\end{equation}
		$$+\sum_{k=1}^n\widetilde{M}_{p}^{kk}\left[(\frac{\varphi''}{\varphi}-\frac{\varphi'^2}{\varphi^2})u_{kk}^2u_k^2+\frac{\varphi'}{\varphi}(u_{kk}^2+\sum_{l=1}^nu_{lkk}u_l)\right]$$
		$$+\sum_{k=1}^n\widetilde{M}_{p}^{kk}\left[\frac{u_{11kk}}{u_{11}}+2\sum_{m>1}\frac{u_{1mk}^2}{u_{11}(u_{11}-u_{mm}+1)}-\frac{u_{11k}^2}{u_{11}^2}\right]=
		:I+II+III.$$
		
		Below we handle each of the terms $I$, $II$ and $III$ separately.
		
		To begin with we recall that $w$ is $p$-plurisubharmonic and hence \newline $\sum_{k=1}^n\widetilde{M}_{p}^{kk}w_{kk}\geq 0$, while
		it follows from Lemma \ref{basicM} and the diagonality of $D^2u$ at $x_0$ that 
		$$\sum_{k=1}^n\widetilde{M}_{p}^{kk}u_{kk}=\tilde{f}$$
		with $\tilde{f}:=f^{1/\binom np}$. As a result
		\begin{equation}\label{I}
			I\geq -\frac{C_1}{w-u}-(1+\delta)\sum_{k=1}^n\widetilde{M}_{p}^{kk}\frac{(w_k-u_k)^2}{(w-u)^2}
		\end{equation}
		for a constant $C_1$ dependent only on $n, p, \delta$ and $sup_{\Om}\tilde{f}$.
		
		Next observe that by differentiating the equation
		$$\M(u)^{\frac1{\binom np}}=\widetilde{M}(u)=f(x,u)^{\frac1{\binom np}},$$
		together with the  the diagonality of $D^2u$ at $x_0$ and concavity of $\widetilde{M}_p$, we obtain
		\begin{equation}\label{secondderivative}
			-C_2(1+u_{11})\leq\tilde{f}_{11}=\sum_{k=1}^n\widetilde{M}_{p}^{kk}u_{kk11}+\sum_{kl,rs}\widetilde{M}_p^{kl,rs}u_{kl1}u_{rs1}
		\end{equation}
		$$=\sum_{k=1}^n\widetilde{M}_{p}^{kk}u_{kk11}+\sum_{l\neq k}\widetilde{M}_p^{lk,kl}u_{kl1}^2+
		\sum_{k,l=1}^n\widetilde{M}_p^{kk,ll}u_{kk1}u_{ll1}$$
		$$\leq \sum_{k=1}^n\widetilde{M}_{p}^{kk}u_{kk11}+2\sum_{l=2}^n\widetilde{M}_p^{1l,l1}u_{11l}^2,$$
		where $C_2$ depends on the $C^{1,1}$ norm of $f$ (or equivalently of $\tilde{f}$ and $f_0$) and the gradient bound for 
		$u$.
		
		Assume now that for an $\e>0$ chosen as in Lemma \ref{cruciallemma} we have
		$$\la_{n-p+1}\geq \e\la_1=\e u_{11}.$$
		
		Then using (\ref{secondderivative}) and concavity of $\widetilde{M}_{p}$ we can bound III as follows
		\begin{equation}\label{III}
			III\geq -C_2\frac{1+u_{11}}{u_{11}}-2\sum_{l=2}^n\widetilde{M}_p^{1l,l1}\frac{u_{11l}^2}{u_{11}}-
		\end{equation}
		$$\sum_{k=1}^n\widetilde{M}_{p}^{kk}[\frac{u_{11k}^2}{u_{11}^2}-2\sum_{m>1}\frac{u_{1mk}^2}{u_{11}(u_{11}-u_{mm}+1)}]\geq  -2C_2-\sum_{k=1}^n\widetilde{M}_{p}^{kk}\frac{u_{11k}^2}{u_{11}^2}.$$
		
		Note that from (\ref{21}) we have
		\begin{equation}\label{l}
			\frac{u_{11k}}{u_{11}}=-\left[(1+\delta)\frac{w_k-u_k}{w-u}+\frac{\varphi'}{\varphi}u_{kk}u_k\right]
		\end{equation}
		and hence from the AM-GM inequality we obtain for a suitable constant $C_\delta>0$ the bound
		$$III\geq -C_3-\frac1{\delta+\delta^2}\sum_{k=1}^n\widetilde{M}_{p}^{kk}(\frac{\varphi'}{\varphi})^2u_{kk}^2u_k^2-
		C_{\delta}\sum_{k=1}^n\widetilde{M}_{p}^{kk}\frac{(w_k-u_k)^2}{(w-u)^2}.$$
		
		Coupling thus the bounds for $I$ and $III$ we obtain
		$$0\geq I+II+III\geq -\frac{C_1}{w-u}-C_4\sum_{k=1}^n\widetilde{M}_{p}^{kk}\frac1{(w-u)^2}$$
		$$+\sum_{k=1}^n\widetilde{M}_{p}^{kk}\left[(\frac{\varphi''}{\varphi}-\frac{1+\delta+\delta^2}{\delta+\delta^2}\frac{\varphi'^2}{\varphi^2})u_{kk}^2u_k^2+\frac{\varphi'}{\varphi}(u_{kk}^2+\sum_{l=1}^nu_{lkk}u_l)\right].$$
		
		Recalling (\ref{phi}) we obtain 
		$$\sum_{k=1}^n\widetilde{M}_{p}^{kk}(\frac{\varphi''}{\varphi}-\frac{1+\delta+\delta^2}{\delta+\delta^2}\frac{\varphi'^2}{\varphi^2})u_{kk}^2u_k^2=0,$$
		while 
		$$\sum_{k=1}^n\widetilde{M}_{p}^{kk}\sum_{l=1}^nu_{lkk}u_l=\sum_{l=1}^n(\tilde{f})_lu_l\geq -C_5$$
		for a constant $C_5$ dependent on $f$ and the gradient bound of $u$.
		
		Thus if $\widetilde{M}$ denotes the sum $\sum_{k=1}^n\widetilde{M}_{p}^{kk}$ we obtain the inequality
		\begin{equation}\label{finallyeasycase}
			0\geq -C_6(1+\frac1{w-u})-\frac{C_4}{(w-u)^2}\widetilde{M}+\frac{\varphi'}{\varphi}\sum_{k=1}^n\widetilde{M}_{p}^{kk}u_{kk}^2. 
		\end{equation}
		
		At this place we use the assumption $\la_{n-p+1}\geq \e\la_1$ to bound \newline $\frac{\varphi'}{\varphi}\sum_{k=1}^n\widetilde{M}_{p}^{kk}u_{kk}^2$ by
		$$\frac{\varphi'}{\varphi}\widetilde{M}_{p}^{n-p+1 n-p+1}u_{n-p+1,n-p+1}^2\geq C_7 \widetilde{M}_{p}^{n-p+1 n-p+1}u_{11}^2.$$
		
		Recall that Lemma \ref{usedingradient} implies that $\widetilde{M}_{p}^{n-p+1 n-p+1}\geq\theta\widetilde{M}$ and thus
		
		$$0\geq C_8\widetilde{M}u_{11}^2-C_6(1+\frac1{w-u})-\frac{C_4}{(w-u)^2}\widetilde{M}.$$
		
		As $\widetilde{M}\geq p$ at $x_0$, by Corollary \ref{below} we obtain 
		$$(w-u)u_{11}=(w-u)\hat{\la}_1\leq C_9\ {\rm at}\ x_0$$
		and hence $(w-u)^{1+\delta}\hat{\la}_1\leq C_{10}$ everywhere by the maximality of $\hat{G}$ at $x_0$.

		\bigskip
		\bigskip
		Assume now that $\la_{n-p+1}\leq \e\la_1$. Following \cite{CW} (see also \cite{Wa2}) we shall apply (\ref{21}) to get rid of the third order term $\frac{u_{111}^2}{u_{11}^2}$ appearing in III and use the concavity of the equation together with Lemma \ref{cruciallemma}
		to get rid of the remaining third order terms. Here the usage of $\hat{\la}_1$ provides us with an additional positive summand which suffices to handle the negative third order terms.
		
		To this end recall that III reads
		\begin{align*}
			&III=\sum_{k=1}^n\widetilde{M}_{p}^{kk}\left[\frac{u_{11kk}}{u_{11}}+2\sum_{m>1}\frac{u_{1mk}^2}{u_{11}(u_{11}-u_{mm}+1)}-\frac{u_{11k}^2}{u_{11}^2}\right]\geq -C_2\frac{1+u_{11}}{u_{11}}+\\
			&2\sum_{m>1}^n\widetilde{M}_{p}^{11}\frac{u_{1m1}^2}{u_{11}(u_{11}-u_{mm}+1)}-2\sum_{l=2}^n\widetilde{M}_p^{1l,l1}\frac{u_{11l}^2}{u_{11}}-\sum_{k=1}^n\widetilde{M}_{p}^{kk}\frac{u_{11k}^2}{u_{11}^2},
		\end{align*}
		where we again used (\ref{III}).
		
		Now we apply (\ref{l}) for $k=1$ and the equivalent equality 
		$$\frac{w_k-u_k}{w-u}=-\frac{1}{1+\delta}(\frac{\varphi'}{\varphi}\sum_{l=1}^nu_{lk}u_l+\frac{u_{11k}}{u_{11}})$$
		for each $k\geq 2$ to obtain
		\begin{equation}\label{1to1}
			-(\frac{u_{111}}{u_{11}})^2\geq -\frac{1}{\delta(1+\delta)}(\frac{\varphi'}{\varphi})^2u_{11}^2u_1^2-(\frac{1}{1-\delta-\delta^2})(1+\delta)^2(\frac{w_1-u_1}{w-u})^2
		\end{equation}
		$$\geq  -\frac{1}{\delta(1+\delta)}(\frac{\varphi'}{\varphi})^2u_{11}^2u_1^2-4(1+\delta)^2(\frac{w_1-u_1}{w-u})^2$$
		and
		\begin{equation}\label{2ton}
			\forall k\geq 2\ \  -(\frac{w_k-u_k}{w-u})^2\geq -\frac{1}{1+\delta}(\frac{u_{11k}}{u_{11}})^2-\frac{1}{\delta(1+\delta)}(\frac{\varphi'}{\varphi})^2u_{kk}^2u_k^2.
		\end{equation}
		
		These two inequalities coupled with the estimate for III yield
		\begin{equation}\label{secondcase}
			0\geq \sum_{k=1}^n\widetilde{M}_{p}^{kk}(1+\delta)\left[\frac{w_{kk}-u_{kk}}{w-u}-\frac{(w_k-u_k)^2}{(w-u)^2}\right]
		\end{equation}
		$$+\sum_{k=1}^n\widetilde{M}_{p}^{kk}\left[(\frac{\varphi''}{\varphi}-\frac{\varphi'^2}{\varphi^2})u_{kk}^2u_k^2+\frac{\varphi'}{\varphi}(u_{kk}^2+\sum_{l=1}^nu_{lkk}u_l)\right]$$
		$$+\sum_{k=1}^n\widetilde{M}_{p}^{kk}\left[\frac{u_{11kk}}{u_{11}}+2\sum_{m>1}\frac{u_{1mk}^2}{u_{11}(u_{11}-u_{mm}+1)}-\frac{u_{11k}^2}{u_{11}^2}\right]
		$$
		$$\geq (1+\delta)\frac{\tilde{f}}{w-u}-(1+\delta)(5+4\delta)\widetilde{M}_{p}^{11}(\frac{w_1-u_1}{w-u})^2$$
		$$+\sum_{k=1}^n\widetilde{M}_{p}^{kk}\left[(\frac{\varphi''}{\varphi}-(1+\frac{1}{\delta(1+\delta)})\frac{\varphi'^2}{\varphi^2})u_{kk}^2u_k^2\right]+\frac{\varphi'}{\varphi}\left[\sum_{k=1}^n\widetilde{M}_{p}^{kk}u_{kk}^2+\sum_{l=1}^nu_l(\tilde{f})_l\right]$$
		$$-2C_2+\sum_{k=2}^n\left[-2\widetilde{M}_{p}^{1k,k1}\frac{u_{11k}^2}{u_{11}}+2\widetilde{M}_{p}^{11}\frac{u_{11k}^2}{u_{11}(u_{11}-u_{kk}+1)}-\frac{\delta+2}{\delta+1}\widetilde{M}_{p}^{kk}\frac{u_{11k}^2}{u_{11}^2}\right].$$

		Observe that  Lemma \ref{cruciallemma} implies that the last term is non negative as long as $\la_1\geq 2+\frac2\delta$ which we can safely assume as otherwise we are through. On the other hand by the choice of $\varphi$ (see (\ref{phi})) the third term is zero.
		
		Thus 
		\begin{equation}\label{finally}
			0\geq- C_{11}(1+\frac1{w-u})-C_{12}\widetilde{M}_{p}^{11}(\frac1{w-u})^2+\frac{\varphi'}{\varphi}\widetilde{M}_{p}^{11}u_{11}^2.
		\end{equation}
		
		Note that if $\frac {\varphi'}{2\varphi}u_{11}^2\leq \frac{C_{12}}{(w-u)^2}$ at $x_0$, then $(w-u)u_{11}$ is under control and the proof is finished. Hence we may assume that the reverse inequality holds which coupled with (\ref{finally}) and Proposition \ref{toadd} yields
		$$0\geq - C_{11}(1+\frac1{w-u})+C_{13}\widetilde{M}_{p}^{11}u_{11}^2\geq  - C_{11}(1+\frac1{w-u})+C_{14}u_{11},$$
		which finally yields the claim in this case.
	\end{proof}
	\section{Liouville theorem}
	In this final section we prove the following result:
	\begin{theorem}
		Let $u$ be a $p$-plurisubharmonic $C^4$ smooth function in $\mathbb R^n$. Assume that for some positive constants $C_1,C_2,C_3$ and $C_4$ we have the bound
		$$C_1|x|^2-C_2\leq u(x)\leq C_3|x|^2+C_4.$$
		If
		$$\M(u)=const>0,$$
		then $u$ is a quadratic polynomial.
	\end{theorem}
	\begin{proof}
		We follow the argument from \cite{HSX} which exploits the interior $C^2$ estimate and a global $C^1$ estimate. In our argument the upper bound for $u$ is utilized to derive the necessary $C^1$ estimate.
		
		Adding a constant if necessary we may assume that $u(0)=0$. We fix a sufficiently large radius $R>max\lbr 1,\sqrt{4C_4}\rbr$. Consider the function
		$$v_R(y):=\frac{u(Ry)-R^2}{R^2}.$$
		
		It is straightforward to check that the function $v_R$ is $p$-plurisubharmonic in the rescaled $y$-variables and satisfies
		\begin{equation}\label{lio1}
			\M(v_R(y))=\M(u(Ry))=const>0.
		\end{equation}
		Consider $\Om_R$- the domain defined as the connected component of
		$$\lbr y\in\mathbb R^n|\ v_R(y)<0\rbr$$
		that contains the origin.
		
		Note that 
		$$\Om_R\subset\lbr y\in\mathbb R^n|\ C_1R^2|y|^2-C_2<R^2\rbr\subset\lbr y\in\mathbb R^n|\ |y|^2\leq \frac{C_2+1}{C_1}\rbr,$$
		hence $\Om_R$ is bounded. 
		
		By definition $v_R<0$ on $\Om_R$ and via the standard maximum principle for Hessian type equations (see \cite{CNS}) comparison with $C|y|^2$ for some $C$ large enough, dependent only on $n$ and $p$,  yields an uniform lower bound for $v$ on $\Om_R$.
		
		Fix a point $y_0\in \Om_R$ and let $y$ be in $B_1(y_0)$. Then, using the upper bound on $u$ we get
		$$v_R(y)=\frac{u(R(y_0+y-y_0))-R^2}{R^2}\leq \frac{2C_3R^2|y_0|^2+2C_3R^2|y-y_0|^2+C_4-R^2}{R^2}$$
		$$\leq 2C_3|y_0|^2+2C_3+C_4-1,$$
		while at the same time
		$$v_R(y)\geq\frac{\frac{C_1}2R^2|y_0|^2-C_1R^2|y-y_0|^2-C_2-R^2}{R^2}\geq C_1-C_2-1.$$
		
		As a result $v_R$ is uniformly bounded in the ball $B_1(y_0)$ and hence by Theorem \ref{gradbound}
		\begin{equation}\label{lio2}
			|Dv_R(y_0)|\leq C
		\end{equation}
		for every point $y_0\in\Om_R$ and a constant $C$ independent on $y_0$.
		
		With this global $C^1$ bound in hand the proof repeats the argument form \cite{HSX}. We include it for the sake of completeness.
		
		The uniform bound (\ref{lio2}) and Theorem \ref{hessbound} result in 
		$$(-v_R(y))^2|D^2v_R(y)|\leq C$$
		in $\Om_R$. Hence on 
		$$\Om'_R:=\lbr y\in\Om_R|\ v_R\leq -\frac12\rbr=\lbr y\in\Om_R|\ u(Ry)\leq\frac12R^2\rbr$$
		the Hessian of $v_R$ is uniformly bounded. But $D^2v_R(y)=(D^2u)(Ry)$ and hence $u$ has a bounded Hessian on
		$$\lbr x\in\mathbb R^n |\ u(x)\leq \frac12R^2\rbr\supset\lbr x\in \mathbb R^n|\ C_3|x|^2+C_4<\frac12R^2\rbr\supset\lbr |x|\leq\frac{R}{2\sqrt{C_3}}\rbr.$$
		
		Finally Evans-Krylov theory (see \cite{GT}) implies that for some $\alpha\in(0,1)$
		$$|D^2u|_{C^{\alpha}(B_{\frac{R}{4\sqrt{C_3}}}(0))}\leq \frac{sup_{B_{\frac{R}{2\sqrt{C_3}}}(0)}|D^2u|}{R^{\alpha}}\rightarrow 0^+$$
		as $R\rightarrow\infty$. This implies that $u$ has to be a quadratic polynomial.
	\end{proof}

	Faculty of Mathematics and Computer Science,  Jagiellonian University 30-348 Krakow, Lojasiewicza 6, Poland;\\ e-mail: {\tt slawomir.dinew@im.uj.edu.pl}
	

\begin{thebibliography}{CNS}
		\bibitem[AO]{ADO20} S. Abja, G. Olive, Local regularity for concave homogeneous complex degenerate elliptic equations dominating the Monge-Ampère equation. \it Ann. Mat. Pura Appl. \rm (4) 201 (2022), no. 2, 561–587.
		\bibitem[Bl1]{Bl96} Z. B\l ocki, The complex Monge-Amp\`ere operator in hyperconvex domains. \it  Ann. Scuola Norm. Sup. Pisa Cl. Sci.\rm\  (4) 23 (1996), no. 4, 721-747.
		\bibitem[CJ]{CJ} J. Chu, H. Jiao, Curvature estimates for a class of Hessian type equations, \it Calc. Var. PDE \rm 60 (2021), no. 3, Paper No. 90, 18 pp.
		\bibitem[CNS]{CNS} L. Caffarelli, L. Nirenberg, J. Spruck, The Dirichlet problem for nonlinear second order elliptic equations, III: Functions of the eigenvalues of the Hessian, \it Acta Math.\rm\ 155 (1985), 261-301.
		\bibitem[CW]{CW} K.-S. Chou, X.-J. Wang, A variational theory of the Hessian equation. \it Comm. Pure Appl. Math. \rm 54 (2001), no 9. 1029-1064.
		\bibitem[De]{De} B. Deng, The Neumann problem for a class of fully nonlinear elliptic partial differential equations, \it preprint,\ \rm arXiv 1903.04231.
		\bibitem[D]{D} S. Dinew, $m$-subharmonic and $m$-plurisubharmonic functions - on two problems of Sadullaev, \it Ann. Fac. Sci. Toulouse Math. \rm (6) 31 (2022), no. 3, 995-1009.  
		\bibitem[Ga]{Ga} L. G\aa rding, An inequality for hyperbolic
		polynomials, \it J. Math. Mech.\rm\ 8 (1959) 957-965.
		\bibitem[GT]{GT} D. Gilbarg, N. Trudinger, Elliptic partial differential equations of second order.
		Reprint of the 1998 edition. Classics in Mathematics. Springer-Verlag, Berlin, 2001. xiv+517 pp.
		\bibitem[HL12]{HL12} F. R.  Harvey, H. B. Lawson,  Geometric plurisubharmonicity and convexity: an introduction. \it Adv. Math.\rm\  230 (2012), no. 4-6, 2428-2456.
		\bibitem[HL13]{HL13} F. R. Harvey, H. B. Lawson, $p$-convexity, $p$-plurisubharmonicity and the Levi problem. \it Indiana Univ. Math. J.\rm\  62 (2013), no. 1, 149-169.
		\bibitem[HL18]{HL1} F. R. Harvey, H. B. Lawson, Tangents to subsolutions- existence and uniqueness, Part I, \it Ann. Fac. Sci. Toulouse Math.\rm\  (6) 27 (2018), no. 4, 777-848.
		\bibitem[HL17]{HL2} F. R. Harvey, H. B. Lawson, Tangents to subsolutions- existence and uniqueness, Part II,
		J. Geom. Anal. 27 (2017), no. 3, 2190-2223.
		\bibitem[HSX]{HSX} Y. He, H. Sheng, N. Xiang, A Liouville type theorem to 2-Hessian equations, preprint, arXiv 1906.10588.  
		\bibitem[K]{K2} S. Ko\l odziej, The complex Monge-Amp\`ere equation and
		pluripotential theory, \it Memoirs Amer. Math. Soc. \rm 178 (2005)
		64p.
		\bibitem[LS]{LS} A.-M. Li, L. Sheng, A Liouville theorem on the PDE $det(f_{i\bar{j}})=1$, \it Math. Z. \rm 297 (2020), 1623-1632.
		\bibitem[Sa]{Sa} A. Sadullaev, Further developments of the pluripotential theory (survey). \it Algebra, complex analysis, and pluripotential theory,\rm 167-182,
		Springer Proc. Math. Stat., (264), Springer, Cham, (2018).
			\bibitem[Sp]{S} J. Spruck, Geometric aspects of the theory of fully nonlinear elliptic equations,  \it Global theory of minimal surfaces, \rm 283-309,
			Clay Math. Proc., 2, Amer. Math. Soc., Providence, RI, 2005. 
		\bibitem[Tr]{Tr} N. Trudinger, Weak solutions of Hessian equations. \it\ 
		Comm. Partial Differential Equations\rm\ 22 (1997), no. 7-8, 1251-1261. 
		\bibitem[TU]{TU} N. Trudinger, J. Urbas, On second derivative estimates for equations of Monge-Amp\`ere type.
		\it Bull. Austral. Math. Soc.\rm\ 30 (1984), no. 3, 321-334.
		\bibitem[TW]{TW} V. Tosatti, B. Weinkove, The Monge-Amp\`ere equation for $(n-1)$-plurisubharmonic functions on a compact K\"ahler manifold,
		\it J. Amer. Math. Soc. \rm\ (30) (2017), no. 2, 311-346.
		\bibitem[W]{Wa2} X.-J. Wang, The $k$-Hessian equation, \it Lect. Not. Math.
		\rm 1977 (2009).
		\bibitem[Y]{Y} Y. Yuan, Global solutions to special Lagrangian equations. \it Proc. Amer. Math. Soc.\rm\  134 (2006), no. 5, 1355-1358.
		\bibitem [Z]{Z} W. Zhou, Representation and regularity for the Dirichlet problem for real and complex degenerate Hessian equations, preprint arXiv 1311.6450.
	\end{thebibliography}
\end{document}